\pdfoutput=1
\documentclass[11pt]{amsart}

\usepackage{amssymb}
\usepackage[pdfauthor={F. Galetto, A.V. Geramita, D.L. Wehlau},
pdftitle={Symmetric complete intersections}]{hyperref}
\usepackage{ytableau}
\usepackage{enumitem}
\usepackage{stmaryrd}
\usepackage{tabulary}
\newcolumntype{K}[1]{>{\centering\arraybackslash}p{#1}}

\theoremstyle{definition} \newtheorem{theorem}{Theorem}[section]
\newtheorem{proposition}[theorem]{Proposition}
\newtheorem{corollary}[theorem]{Corollary}
\newtheorem{lemma}[theorem]{Lemma}

\newtheorem{example}[theorem]{Example}

\numberwithin{table}{section} \numberwithin{figure}{section}

\renewcommand{\geq}{\geqslant} \renewcommand{\leq}{\leqslant}
 
 \newcommand{\ZZ}{\mathbb{Z}}
\newcommand{\bigchi}{\raisebox{2pt}{\large$\chi$}}

\title{Symmetric complete intersections}

\date{\today}

\author{Federico Galetto}
\address{McMaster University\\1280 Main St W\\Hamilton Hall 407\\
  Hamilton, ON, L8S 4K1\\Canada}
\email{galettof@math.mcmaster.ca}
\urladdr{http://math.galetto.org}

\author{Anthony V.~Geramita}

\author{David L.~Wehlau}
\address{Department of Mathematics and Computer Science \\
 Royal Military College, Kingston, ON, K7K 7B4, Canada}
\email{wehlau@rmc.ca}

\begin{document}

\begin{abstract}
  We consider complete intersection ideals in a polynomial ring over a
  field of characteristic zero that are stable under the action of
  the symmetric group permuting the variables.  We determine the
  possible representation types for these ideals, and describe formulas
  for the graded characters of the corresponding quotient rings.
\end{abstract}

\maketitle{}

\section{Introduction}
\label{sec:introduction}

In recent work \cite{1407.7228}, the last two authors, together with
A.~Hoefel, analyzed a family of artinian Gorenstein algebras that are
also graded representations of the symmetric group. Taking advantage
of the additional structure, they were able to compute the graded
character of the algebra, from which the Hilbert series of the algebra
can be readily obtained.  Inspired by this result, we apply similar
methods to a different class of Gorenstein algebras, namely complete
intersections.

Consider the polynomial ring $R = \Bbbk [x_1,\ldots,x_n]$ and let
$I\subseteq R$ be a homogeneous complete intersection ideal.  Such an
ideal is generated by a regular sequence $f_1,\ldots,f_r$, with
$r\leqslant n$.  Many algebraic properties of $I$, such as its Hilbert
series and its graded Betti numbers, are completely determined by the
degrees of $f_1,\ldots,f_r$.  By Nakayama's lemma, these numbers are
the same as the degrees of the elements in any homogeneous basis of
the $\Bbbk$-vector space $I/\mathfrak{m}I$, where
$\mathfrak{m} = (x_1,\ldots,x_n)$.

Now let the symmetric group $\mathfrak{S}_n$ act on $R$ by permuting
the variables.  We consider complete intersection ideals
$I\subseteq R$ that are stable under the action of $\mathfrak{S}_n$.
In this case, the action descends to the quotient $I/\mathfrak{m}I$
making it a finite dimensional representation of $\mathfrak{S}_n$.
Therefore the ideal $I$ is further parametrized by the isomorphism
type of $I/\mathfrak{m}I$.

We assume $\Bbbk$ has characteristic zero, so $I/\mathfrak{m}I$
decomposes into a direct sum of irreducible representations.  Based on
this decomposition, we prove in Proposition \ref{pro:1} that
$I/\mathfrak{m}I$ belongs to one of four families of isomorphism
types.  We further examine these families in \S~\ref{sec:resol-char},
providing examples of each and determining formulas for the graded
characters of the quotient $R/I$.

An important question remains open: for a fixed isomorphism type of
$I/\mathfrak{m}I$, in what degrees can the generators of $I$ occur?
We plan to address this question in a future paper \cite{GGW16}.

All three authors gratefully acknowledge the partial support of NSERC
for this work.

\section{Classification of \texorpdfstring{$\mathfrak{S}_n$}{Sn}-stable complete
  intersection ideals}
\label{sec:class-mathfr-stable}

Throughout this paper, we consider the standard graded polynomial ring
$R := \Bbbk [x_1,\ldots,x_n]$, where $\Bbbk$ is a field of
characteristic zero.  All ideals we consider are homogeneous, unless
noted otherwise.  An ideal $I \subseteq R$ is called a \emph{complete
  intersection ideal} if it is generated by a regular sequence
\cite[Ch.~17]{MR1322960}.  Note that regular sequences in $R$ have
length at most $n$.  If $I \subseteq R$ is a complete intersection
ideal, then every minimal generating set of $I$ is a regular sequence.

The symmetric group $\mathfrak{S}_n$ acts linearly on
$R_1 = \langle x_1,\ldots,x_n \rangle$ by permuting the variables.
This action extends naturally to all of $R$.  As a result, $R$ is a
graded representation of $\mathfrak{S}_n$ over $\Bbbk$.

We will use some standard notation and ideas which can be found, for
instance, in the book \cite{MR1824028}. We say
$\lambda = (\lambda_1,\ldots,\lambda_r)$ is a \emph{partition} of $n$,
where the $\lambda_i$ are positive integers
$\lambda_1 \geq \ldots \geq \lambda_r \geq 1$, if
$\sum_{i=1}^r \lambda_i = n$. In this case we write $\lambda\vdash n$.
We use the shorthand notation $\lambda_i^{m_i}$ to indicate that the
integer $\lambda_i$ appears $m_i$ times in $\lambda$. For example, we
write $(2^2,1^3)$ to denote the partition $(2,2,1,1,1)$.  We say the
partition $\lambda = (\lambda_1,\ldots,\lambda_r)$ contains the
partition $\mu = (\mu_1,\ldots,\mu_s)$ if $r\geqslant s$ and
$\lambda_i \geqslant \mu_i$ for all $i\leqslant s$.

If $\lambda = (\lambda_1,\ldots,\lambda_r)$ is a partition of $n$,
then the \emph{Young diagram} of $\lambda$ is an array of $n$ boxes
having $r$ left-justified rows and with the $i$-th row containing
$\lambda_i$ boxes.  A \emph{Young tableau} with shape $\lambda$ is a
bijective filling of the boxes of the Young diagram of $\lambda$ with
the numbers $1,2,\ldots ,n$.  A Young tableau is called
\emph{standard} if the entries of each row form an increasing sequence
from left to right, and the entries of each column form an increasing
sequence from top to bottom.

As is well-known, the irreducible representations of the symmetric
group $\mathfrak{S}_n$, over a field of charcteristic zero, are in
one-to-one correspondence with the partitions of $n$
\cite[Thm.~2.4.6]{MR1824028}.  Denote by $S^\lambda$ the irreducible
representation associated with the partition $\lambda$.  The standard
Young tableaux with shape $\lambda$ form a basis of $S^\lambda$.  If
$T$ is a standard tableau with shape $\lambda$ and the entries $i,j$
belong to the same column of $T$, then the transposition
$(i\ j) \in \mathfrak{S}_n$ acts on $T$ by:
\begin{equation*}
  (i\ j) T = -T.
\end{equation*}
We do not specify the action of other permutations on $T$ since, just
using the above equation, we can describe a property of the ideals of
$R$ generated by a single irreducible representation $S^\lambda$.

\begin{lemma}
  \label{lem:2}
  Let $\lambda$ be a partition of $n$ and let
  $\varphi \colon S^\lambda \to R_d$ be a non-zero map of
  $\mathfrak{S}_n$-representations for some $d\in \ZZ$. If $T$ is a
  standard tableau with shape $\lambda$ containing entries $i$ and $j$
  in the same column, then $\varphi (T)$ is a polynomial divisible by
  $x_i - x_j$.
\end{lemma}

\begin{proof}
  Note that Schur's Lemma implies that the map $\varphi$ is injective
  (see for example \cite[Thm.~1.6.5]{MR1824028}).

  Let $p = \varphi (T)$. Since $\varphi$ is
  $\mathfrak{S}_n$-equivariant, we have
  \begin{equation*}
    (i\ j) p = \varphi ((i\ j) T) = \varphi (-T) = -\varphi (T) = -p.
  \end{equation*}
  Therefore $p$ is divisible by $x_i - x_j$.
\end{proof}

\begin{corollary}
  \label{cor:1}
  Assume $n\geqslant 5$. Let $\lambda$ be a partition containing
  $(2,2)$ and let $\varphi \colon S^\lambda \to R_d$ be a non-zero map
  of $\mathfrak{S}_n$-representations for some $d\in \ZZ$. If $I$ is
  the ideal of $R$ generated by the image of $\varphi$, then $I$ is
  not a complete intersection.
\end{corollary}

\begin{proof}
  Since $\lambda$ contains $(2,2)$ and $n\geqslant 5$, there are
  standard tableaux $T_1$ and $T_2$ with shape $\lambda$ containing
  the squares
  \begin{center}
    \ytableausetup{boxsize=normal}
    \begin{ytableau}
      1 & 3 \\
      2 & 4
    \end{ytableau}
    \qquad\text{and}\qquad
    \begin{ytableau}
      1 & 3 \\
      2 & 5
    \end{ytableau}
  \end{center}
  respectively.  Let $p_1 = \varphi (T_1)$ and $p_2 = \varphi (T_2)$.
  By Lemma \ref{lem:2}, both $p_1$ and $p_2$ are divisible by
  $x_1 - x_2$.  Since the standard tableaux on $\lambda$ form a basis
  of $S^\lambda$ and since $\varphi$ is injective, the images of the
  standard tableaux under $\varphi$ form a minimal generating set of
  the ideal $I$. Both $p_1$ and $p_2$ are part of this minimal
  generating set. However, $p_1$ and $p_2$ have $x_1 - x_2$ as a
  common factor; therefore our minimal generating set of $I$ is not a
  regular sequence. Therefore $I$ is not a complete intersection.
\end{proof}

A partition of the form $(a,1^b)$ is called a \emph{hook}.  A
partition $\lambda$ is a hook if and only if it does not contain
$(2,2)$. Using this terminology, Corollary \ref{cor:1} implies that an
irreducible $S^\lambda \subseteq R$ generating a complete intersection
ideal must be a hook. However, not all hooks generate complete
intersections.

\begin{corollary}
  \label{cor:2}
  Let $\lambda$ be a hook partition of $n$ containing $(2,1,1)$ and
  let $\varphi \colon S^\lambda \to R_d$ be a non-zero map of
  $\mathfrak{S}_n$-representations for some $d\in \ZZ$. If $I$ is the
  ideal of $R$ generated by the image of $\varphi$, then $I$ is not a
  complete intersection.
\end{corollary}

\begin{proof}
  Since $\lambda$ contains $(2,1,1)$, there are standard tableaux
  $T_1$ and $T_2$ with shape $\lambda$ containing the hooks
  \begin{center}
    \ytableausetup{boxsize=normal}
    \begin{ytableau}
      1 & 4\\
      2\\
      3
    \end{ytableau}
    \qquad\text{and}\qquad
    \begin{ytableau}
      1 & 3 \\
      2\\
      4
    \end{ytableau}
  \end{center}
  respectively.  Let $p_1 = \varphi (T_1)$ and $p_2 = \varphi (T_2)$.
  By Lemma \ref{lem:2}, both $p_1$ and $p_2$ are divisible by
  $x_1 - x_2$.  Thus $I$ is not a complete intersection.
\end{proof}

As a consequence of Corollaries \ref{cor:1} and \ref{cor:2}, the only
irreducible representations $S^\lambda \subseteq R$ that can generate
complete intersections are:
\begin{enumerate}[label=\alph*)]
\item the trivial representation $S^{(n)}$ (1-dimensional),
\item the alternating representation $S^{(1^n)}$ (1-dimensional),
\item the standard representation $S^{(n-1,1)}$ ($(n-1)$-dimensional),
\item and, for $n=4$, the exceptional case $S^{(2,2)}$
  (2-dimensional).
\end{enumerate}
Moreover, there are complete intersections of each of these four
types.

Note that trivial representations correspond to symmetric polynomials
while alternating representations correspond to alternating
polynomials, i.e., products of the Vandermonde determinant on
$x_1,\ldots,x_n$ with a symmetric polynomial.

Consider an ideal $I\subseteq R$ and let
$\mathfrak{m} = (x_1,\ldots,x_n)$ be the irrelevant ideal of $R$.  By
the graded version of Nakayama's Lemma (see
\cite[Cor.~4.8]{MR1322960}), the homogeneous polynomials
$f_1,\ldots,f_t$ form a minimal generating set of $I$ if and only if
their residue classes form a basis of $I/\mathfrak{m}I$.  Thus a basis
of $I/\mathfrak{m}I$ can be lifted to a minimal generating set of $I$.
Equivalently, one can define a section $I/\mathfrak{m}I \to I$ of the
projection $I\to I/\mathfrak{m}I$, and lift a basis of
$I/\mathfrak{m}I$ along this section.  Note that liftings and sections
are not unique.

Now suppose $I$ is $\mathfrak{S}_n$-stable.  Since $\mathfrak{m}I$ is
clearly $\mathfrak{S}_n$-stable, the action of $\mathfrak{S}_n$
descends to $I/\mathfrak{m}I$.  Hence the projection
$I \to I/\mathfrak{m}I$ is a map of graded $\mathfrak{S}_n$-representations
and it admits degree-preserving $\mathfrak{S}_n$-equivariant sections.  Fix one such
section and let $V$ denote its image.  The vector space $V$ is a
graded $\mathfrak{S}_n$-stable subspace of $I$ and, by our previous
observations, any basis of $V$ is also a minimal generating set of
$I$.  We call $V$ an \emph{$\mathfrak{S}_n$-equivariant lift} of
$I/\mathfrak{m}I$ to $R$.  Observe that if $V$ and $V'$ are
$\mathfrak{S}_n$-equivariant lifts of $I/\mathfrak{m}I$ to $R$, then
$V \cong V'$ as graded $\mathfrak{S}_n$-representations.

Next we consider ideals that are generated by more than one
irreducible representation.

\begin{corollary}
  \label{cor:3}
  Let $I$ be an $\mathfrak{S}_n$-stable ideal of $R$ and let $V$ be an
  $\mathfrak{S}_n$-equivariant lift of $I/\mathfrak{m}I$ to $R$. If
  $V$ contains at least two (potentially isomorphic) non-trivial
  irreducible representations of $\mathfrak{S}_n$, then $I$ is not a
  complete intersection.
\end{corollary}

\begin{proof}
  By the hypothesis, $V = W_1 \oplus W_2 \oplus W_3$, where
  \begin{itemize}
  \item $W_1$ is the image of a non-zero $\mathfrak{S}_n$-equivariant
    map $\varphi \colon S^\lambda \to R_{d_1}$ for some $d_1 \in \ZZ$
    and $\lambda \neq (n)$;
  \item $W_2$ is the image of a non-zero $\mathfrak{S}_n$-equivariant
    map $\psi \colon S^\mu \to R_{d_2}$ for some $d_2 \in \ZZ$ and
    $\mu \neq (n)$;
  \item $W_3$ is the $\mathfrak{S}_n$-stable complement of
    $W_1\oplus W_2$ in $V$.
  \end{itemize}
  Observe that we may have $W_3 = 0$.

  Since $\lambda$ and $\mu$ are both different from $(n)$, there are
  standard tableaux $T$ with shape $\lambda$ and $U$ with shape $\mu$
  both containing the entries 1 and 2 in the same column. By Lemma
  \ref{lem:2}, $\varphi (T)$ and $\psi (U)$ are both divisible by
  $x_1-x_2$. Thus $I$ is not a complete intersection.
\end{proof}

Corollary \ref{cor:3} immediately implies the following.

\begin{proposition}
  \label{pro:1}
  Suppose $I\subseteq R$ is an $\mathfrak{S}_n$-stable complete
  intersection ideal. Then $I/\mathfrak{m}I$ is isomorphic to one of
  the following representations:
  \begin{enumerate}[label=\Roman*.]
  \item $\bigoplus^m S^{(n)}$, with $1\leqslant m\leqslant n$, i.e., a
    direct sum of up to $n$ trivial representations;
  \item $S^{(1^n)} \oplus (\bigoplus^m S^{(n)})$, with
    $0\leqslant m\leqslant n-1$, i.e., a direct sum of an alternating
    representation and up to $n-1$ trivial representations;
  \item $S^{(n-1,1)} \oplus (\bigoplus^m S^{(n)})$, with
    $0\leqslant m\leqslant 1$, i.e., a direct sum of a standard
    representation and up to one trivial representation;
  \item for $n=4$, $S^{(2,2)} \oplus (\bigoplus^m S^{(4)})$, with
    $0\leqslant m\leqslant 2$, i.e., a direct sum of the irreducible
    representation $S^{(2,2)}$ and up to two trivial representations.
  \end{enumerate}
\end{proposition}

\section{Graded characters}
\label{sec:resol-char}

Let $I \subseteq R$ be an $\mathfrak{S}_n$-stable complete
intersection ideal.  In this section, we obtain formulas for the
graded character of $R/I$.  Recall that the Koszul complex on the
generators of $I$ is a minimal free resolution of the quotient $R/I$
\cite[Cor.~17.5]{MR1322960}. Since $I$ is $\mathfrak{S}_n$-stable, the
action of $\mathfrak{S}_n$ extends to the free modules in the Koszul
complex and commutes with the differentials.

Throughout this section, we will denote by
$\langle g_1,\ldots,g_r\rangle$ the graded $\Bbbk$-vector space
spanned by the polynomials $g_1,\ldots,g_r$, as opposed to the ideal
generated by the same polynomials which we denote by
$(g_1,\ldots,g_r)$.  Unless otherwise noted, tensor products
are taken over $\Bbbk$.

We begin by describing the graded character of the polynomial ring
$R$, which will play a fundamental role in our results.  Then we will
analyze the four cases of Proposition \ref{pro:1} separately.

\subsection{The graded character of the polynomial ring}
\label{sec:grad-char-polyn}

Let us begin by introducing our notation for characters; the reader
may consult \cite[\S~1.8]{MR1824028} for the basic definitions.
Given a representation $V$ of $\mathfrak{S}_n$ over $\Bbbk$, we denote
its character by $\bigchi_V$.  Recall that $\bigchi_V$ is a class
function on $\mathfrak{S}_n$.  The set $\mathcal{C}$ of class
functions on $\mathfrak{S}_n$ is a commutative ring with pointwise
addition and multiplication.  We write $\bigchi \bigchi'$ for
the product of two class functions $\bigchi,\bigchi' \in \mathcal{C}$.
For $\lambda \vdash n$, let $\bigchi^\lambda = \bigchi_{S^\lambda}$
denote the irreducible character of $\mathfrak{S}_n$ corresponding to
the irreducible representation $S^\lambda$.  The irreducible
characters $\bigchi^\lambda$ form a basis of $\mathcal{C}$ as a
$\Bbbk$-vector space.

If $V = \bigoplus_{d\in\ZZ} V_d$ is a graded representation of
$\mathfrak{S}_n$ over $\Bbbk$, then $\bigchi_V$ will denote its graded
character; namely,
\begin{equation*}
  \bigchi_V = \sum_{d\in \ZZ} \bigchi_{V_d} t^d.
\end{equation*}
The graded character $\bigchi_V$ is an element of the power series
ring $\mathcal{C} \llbracket t\rrbracket$.  With the exception of the
irreducible characters $\bigchi^\lambda$, all characters we consider
in this section are graded.

Let us recall some facts about the invariant theory of
$\mathfrak{S}_n$, which can be found, for example, in
\cite[Thm.~4.1]{MR2590895}.  We denote by $R^{\mathfrak{S}_n}$ the
$\mathfrak{S}_n$-invariant subring of $R$. Denote by
$e_i \in R^{\mathfrak{S}_n}$ the $i$-th elementary symmetric
polynomial. The coinvariant algebra of $\mathfrak{S}_n$ on $R$,
denoted by $R_{\mathfrak{S}_n}$, is the quotient
$R/(e_1,\ldots,e_n)$. Then we have that:
\begin{enumerate}[label=(\alph*)]
\item\label{item:2} $R^{\mathfrak{S}_n} = \Bbbk [e_1,\ldots,e_n]$; in
  particular, $R^{\mathfrak{S}_n}$ is a polynomial ring;
\item\label{item:1}
  $R \cong R^{\mathfrak{S}_n} \otimes R_{\mathfrak{S}_n}$ as graded
  representations of $\mathfrak{S}_n$.
\end{enumerate}
Using \ref{item:2}, we can write the Hilbert series of
$R^{\mathfrak{S}_n}$ as
\begin{equation*}
  \frac{1}{\prod_{j=1}^n (1-t^j)}.
\end{equation*}
By definition, $\mathfrak{S}_n$ acts trivially on
$R^{\mathfrak{S}_n}$; thus the graded character of the ring of
invariants is
\begin{equation*}
  \bigchi_{R^{\mathfrak{S}_n}} = \frac{\bigchi^{(n)}}{\prod_{j=1}^n (1-t^j)}.
\end{equation*}
Combining fact \ref{item:1} with the equation above, we obtain the
formula
\begin{equation*}
  \bigchi_R = \bigchi_{R^{\mathfrak{S}_n}} \bigchi_{R_{\mathfrak{S}_n}} =
  \frac{\bigchi_{R_{\mathfrak{S}_n}}}{\prod_{j=1}^n (1-t^j)}.
\end{equation*}

To describe the graded character of $R_{\mathfrak{S}_n}$ we need some
additional tools from the theory of symmetric functions which can be
found, for example, in \cite[\S~I.3, \S~III.2,
\S~III.6]{MR1354144}.  The \emph{Schur polynomials}, denoted by
$s_\lambda (x_1,\ldots,x_n)$, are a family of symmetric polynomials.
The polynomials $s_\lambda (x_1,\ldots,x_n)$ with $\lambda \vdash d$
form a basis of the graded component $(R^{\mathfrak{S}_n})_d$ as a
$\Bbbk$-vector space. Therefore, as $\lambda$ varies over all
partitions of all natural numbers, the Schur polynomials form a basis
of $R^{\mathfrak{S}_n}$ as a $\Bbbk$-vector space.  By extending
scalars to the ring $\Bbbk [t]$, we deduce that the Schur polynomials
form a basis of $R^{\mathfrak{S}_n} [t]$ as a $\Bbbk [t]$-module.
Another family of symmetric polynomials are the \emph{Hall-Littlewood
  polynomials}, denoted by $P_\lambda (x_1,\ldots,x_n;t)$, which are
indexed by a partition and a parameter $t$.  As $\lambda$ varies over
all partitions of all natural numbers, the Hall-Littlewood polynomials
form another basis of $R^{\mathfrak{S}_n} [t]$ as a
$\Bbbk [t]$-module.  The transition between these two bases of
$R^{\mathfrak{S}_n} [t]$ is encoded by the formula
\begin{equation*}
  s_\lambda (x_1,\ldots,x_n) = \sum_{\mu \vdash m} K_{\lambda,\mu} (t)
  P_\mu (x_1,\ldots,x_n;t),
\end{equation*}
where $\lambda \vdash m\in\mathbb{N}$. The coefficients
$K_{\lambda,\mu} (t)$ are in $\ZZ [t]$ and are known as
\emph{Kostka-Foulkes} polynomials. It is worth mentioning that the
polynomials $K_{\lambda,\mu} (t)$ can be computed using a
combinatorial procedure due to A.~Lascoux and M.P.~Sch\"utzenberger (see
\cite[Ch.~III (6.5) (i)]{MR1354144}).  For our purposes, we need to
use a slightly modified version of $K_{\lambda,\mu} (t)$ which we will
denote by $\tilde{K}_{\lambda,\mu} (t)$. For
$\mu = (\mu_1,\ldots,\mu_r) \vdash n$, introduce the integer
\begin{equation*}
  n(\mu) = \sum_{i=1}^r (i-1)\mu_i,
\end{equation*}
then define
\begin{equation*}
  \tilde{K}_{\lambda,\mu} (t) = K_{\lambda,\mu} (1/t) t^{n(\mu)}.
\end{equation*}

Finally, we can write a formula for the graded character of
$R_{\mathfrak{S}_n}$:
\begin{equation*}
  \bigchi_{R_{\mathfrak{S}_n}} = \sum_{\lambda \vdash n}
  \bigchi^\lambda \tilde{K}_{\lambda, (1^n)} (t).
\end{equation*}
An elementary proof can be found in \cite{MR1168926}.

\begin{example}
  \label{exa:1}
  Let $n=4$, so $R=\Bbbk [x_1,\ldots,x_4]$.  The polynomials
  $\tilde{K}_{\lambda,(1^4)} (t)$ for $\lambda \vdash 4$ are:
  \begin{align*}
    &\tilde{K}_{(4),(1^4)} (t) = 1,&  &\tilde{K}_{(3,1),(1^4)} (t) = t+t^2+t^3,\\
    &\tilde{K}_{(2,2),(1^4)} (t) = t^2+t^4,&  &\tilde{K}_{(2,1^2),(1^4)} (t) = t^3+t^4+t^5,\\
    &\tilde{K}_{(1^4),(1^4)} (t) = t^6;& &
  \end{align*}
  this can be deduced from the table of polynomials
  $K_{\lambda,\mu} (t)$ for $n=4$ in
  \cite[p.~239]{MR1354144}. Therefore we have
  \begin{equation}
    \label{eq:5}
    \begin{split}
      \bigchi_{R_{\mathfrak{S}_4}} &= \bigchi^{(4)} + \bigchi^{(3,1)} t + (\bigchi^{(3,1)} + \bigchi^{(2,2)})t^2 + (\bigchi^{(3,1)} + \bigchi^{(2,1^2)})t^3\\
      &\quad + (\bigchi^{(2,2)} + \bigchi^{(2,1^2)})t^4 +
      \bigchi^{(2,1^2)} t^5 + \bigchi^{(1^4)} t^6,
    \end{split}
  \end{equation}
  and
  \begin{equation*}
    \begin{split}
      \bigchi_R &= \frac{\bigchi_{R_{\mathfrak{S}_4}}}{\prod_{j=1}^4
        (1-t^j)}
      = \bigchi_{R_{\mathfrak{S}_4}} (1+t+2t^2+3t^3+ 5t^4 + 6t^5 + \ldots)  =\\
      &= \bigchi^{(4)} + (\bigchi^{(4)} + \bigchi^{(3,1)}) t + (2\bigchi^{(4)} + 2\bigchi^{(3,1)} + \bigchi^{(2,2)})t^2 +\\
      &\quad +(3\bigchi^{(4)} + 4\bigchi^{(3,1)} + \bigchi^{(2,2)} + \bigchi^{(2,1^2)})t^3 + \\
      &\quad +(5\bigchi^{(4)} + 6\bigchi^{(3,1)} + 3\bigchi^{(2,2)} +
      2\bigchi^{(2,1^2)})t^4 + \ldots
    \end{split}
  \end{equation*}
\end{example}


\subsection{Case I: trivial representations}
\label{sec:triv-repr}

\begin{theorem}
  \label{thm:1}
  Let $I\subseteq R$ be an $\mathfrak{S}_n$-stable complete
  intersection ideal.  Assume $I/\mathfrak{m}I$ is isomorphic to a
  direct sum of $m$ trivial representations in degrees
  $c_1,\ldots,c_m$.  Then the graded character of $R/I$ is
  \begin{equation*}
    \bigchi_{R/I} = \bigchi_R \prod_{i=1}^m (1-t^{c_i}).
  \end{equation*}
\end{theorem}

\begin{proof}
  Let $V$ be an equivariant lift of $I/\mathfrak{m}I$ to $R$.  In this
  case, $V = \langle f_1, \ldots, f_m\rangle$, where each $f_i$ is a
  homogeneous symmetric polynomial and $\deg(f_i) = c_i$.  Moreover,
  $I$ is minimally generated by the regular sequence
  $f_1, \ldots, f_m$.

  The proof is by induction on $m$. If $m=1$, then we have a short exact sequence
  \begin{equation*}
    0 \to R(-c_1) \xrightarrow{\cdot f_1} R \to R/I \to 0
  \end{equation*}
  of $R$-modules and $\mathfrak{S}_n$-representations.
  By the additivity of characters along exact sequences, we obtain
  \begin{equation*}
    \bigchi_{R/I} = \bigchi_R - \bigchi_{R(-c_1)} =
    \bigchi_R - \bigchi_R t^{c_1} = \bigchi_R (1-t^{c_1}),
  \end{equation*}
  which proves the base case of the induction.

  If $m>1$, then we have a short exact sequence
  \begin{equation*}
    0 \to R/(f_1,\ldots,f_{m-1}) (-c_m) \xrightarrow{\cdot f_m}
    R/(f_1,\ldots,f_{m-1}) \to R/I \to 0
  \end{equation*}
  of $R$-modules and $\mathfrak{S}_n$-representations.
  Using the inductive hypothesis, we know that
  \begin{equation*}
    \bigchi_{R/(f_1,\ldots,f_{m-1})} = \bigchi_R \prod_{i=1}^{m-1} (1-t^{c_i}).
  \end{equation*}
  Finally, by additivity of characters, we have
  \begin{equation*}
    \bigchi_{R/I} = \bigchi_R \prod_{i=1}^{m-1} (1-t^{c_i}) -
    \bigchi_R \prod_{i=1}^{m-1} (1-t^{c_i}) t^{c_m} =
    \bigchi_R \prod_{i=1}^m (1-t^{c_i}).
  \end{equation*}
\end{proof}

\begin{example}
  \label{exa:2}
  Let $n=4$, so $R=\Bbbk [x_1,\ldots,x_4]$.  Consider the ideal
  $I \subseteq R$ generated by the regular sequence
  $e_1^3,e_1^2-e_2,e_3,e_4$.

  By Theorem \ref{thm:1}, we have
  \begin{align*}
    \bigchi_{R/I} &=\bigchi_R (1-t^2)(1-t^3)^2(1-t^4)=\\
                  &=\frac{\bigchi_{R_{\mathfrak{S}_4}}}{\prod_{j=1}^4 (1-t^j)} (1-t^2)(1-t^3)^2(1-t^4) =\\
                  &=\bigchi_{R_{\mathfrak{S}_4}}(1+t+t^2).
  \end{align*}
  Using equation \eqref{eq:5}, we obtain
  \begin{equation*}
    \begin{split}
      \bigchi_{R/I} &=
      \bigchi^{(4)} + (\bigchi^{(4)}+\bigchi^{(3,1)})t+(\bigchi^{(4)}+2\bigchi^{(3,1)}+\bigchi^{(2,2)})t^2\\
      &\quad +(3\bigchi^{(3,1)}+\bigchi^{(2,2)}+\bigchi^{(2,1^2)})t^3+(2\bigchi^{(3,1)}+2\bigchi^{(2,2)}+2\bigchi^{(2,1^2)})t^4\\
      &\quad +(\bigchi^{(3,1)}+\bigchi^{(2,2)}+3\bigchi^{(2,1^2)})t^5+(\bigchi^{(2,2)}+2\bigchi^{(2,1^2)}+\bigchi^{(1^4)})t^6\\
      &\quad +(\bigchi^{(2,1^2)}+\bigchi^{(1^4)})t^7 +
      \bigchi^{(1^4)}t^8.
    \end{split}
  \end{equation*}
\end{example}

\subsection{Case II: one alternating representation}
\label{sec:altern-repr}

\begin{theorem}
  \label{thm:2}
  Let $I\subseteq R$ be an $\mathfrak{S}_n$-stable complete
  intersection ideal.  Assume $I/\mathfrak{m}I$ is isomorphic to a
  direct sum of one alternating representation in degree $d$ and $m$
  trivial representations in degrees $c_1,\ldots,c_m$.  Then the
  graded character of $R/I$ is
  \begin{equation*}
    \bigchi_{R/I} = \bigchi_R (\bigchi^{(n)} -\bigchi^{(1^n)} t^d)
    \prod_{i=1}^m (1-t^{c_i}).
  \end{equation*}
\end{theorem}

\begin{proof}
  Let $V$ be an equivariant lift of $I/\mathfrak{m}I$ to $R$.  In this
  case, $V = \langle g,f_1,\ldots,f_m\rangle$, where $g$ is a
  homogeneous alternating polynomial of degree $d$, and 
  $f_i$ is a homogeneous symmetric polynomial of degree $c_i$.

  We have a short exact sequence
  \begin{equation*}
    0 \rightarrow R (-d) \otimes S^{(1^n)}
    \xrightarrow{\cdot g} R
    \rightarrow R/(g) \rightarrow 0
  \end{equation*}
  of $R$-modules and $\mathfrak{S}_n$-representations. The tensor
  product with the alternating representation is needed make
  multiplication by $g$ into a map of
  $\mathfrak{S}_n$-representations, because $g$ is alternating.
  
  By the additivity of characters, we get
  \begin{equation*}
    \begin{split}
      \bigchi_{R/(g)}
      &= \bigchi_R - \bigchi_{R(-d) \otimes S^{(1^n)}} =\\
      &= \bigchi_R - \bigchi_{R(-d)} \bigchi^{(1^n)} =\\
      &= \bigchi_R \bigchi^{(n)} - \bigchi_R t^d \bigchi^{(1^n)} =\\
      &=\bigchi_R (\bigchi^{(n)} - \bigchi^{(1^n)} t^d).
    \end{split}
  \end{equation*}
  Now the $f_i$ can be modded out one at a time and the proof follows
  the argument used for Theorem \ref{thm:1}.
\end{proof}

\begin{example}
  \label{exa:3}
  Let $n=4$, so $R=\Bbbk [x_1,\ldots,x_4]$.  Consider the ideal
  $I=(e_1^2,e_2,e_3,v) \subseteq R$, where $v$ denotes the Vandermonde
  determinant
  \begin{equation*}
    \prod_{1\leqslant i<j\leqslant4} (x_i - x_j).
  \end{equation*}
  It can be verified, either by hand or using computer algebra
  software, that $e_1^2,e_2,e_3,v$ is a regular sequence.

  Using Theorem \ref{thm:2} and the formula for
  $\bigchi_{R_{\mathfrak{S}_4}}$ in equation \eqref{eq:5}, we have
  \begin{equation*}
    \begin{split}
      \bigchi_{R/I} &= \bigchi_R (\bigchi^{(4)}-\bigchi^{(1^4)}
      t^6)(1-t^2)^2(1-t^3) =\\
      &= \frac{\bigchi_{R_{\mathfrak{S}_4}}}{\prod_{i=1}^4 (1-t^i)} (\bigchi^{(4)}-\bigchi^{(1^4)} t^6) (1-t^2)^2(1-t^3)=\\
      &=\bigchi^{(4)} +(\bigchi^{(4)}+\bigchi^{(3,1)})t + (2\bigchi^{(3,1)}+\bigchi^{(2,2)})t^2\\
      &\quad +(2\bigchi^{(3,1)} + \bigchi^{(2,2)} + \bigchi^{(2,1^2)})t^3\\
      &\quad +(\bigchi^{(4)} + \bigchi^{(3,1)} + \bigchi^{(2,2)} + 2\bigchi^{(2,1^2)})t^4\\
      &\quad +(\bigchi^{(4)} + \bigchi^{(3,1)} +\bigchi^{(2,2)} + 2\bigchi^{(2,1^2)})t^5\\
      &\quad +(2\bigchi^{(3,1)} + \bigchi^{(2,2)} + \bigchi^{(2,1^2)})t^6\\
      &\quad +(2\bigchi^{(3,1)} +\bigchi^{(2,2)})t^7 +(\bigchi^{(4)} +
      \bigchi^{(3,1)})t^8 + \bigchi^{(4)}t^9.
    \end{split}
  \end{equation*}
\end{example}

\subsection{Case III: one standard representation}
\label{sec:standard-repr}

\begin{theorem}
  \label{thm:3}
  Let $I\subseteq R$ be an $\mathfrak{S}_n$-stable complete
  intersection ideal.  Assume $I/\mathfrak{m}I$ is a direct sum of one
  standard representation in degree $d$ and $m$ trivial
  representations in degrees $c_1,\ldots,c_m$.  Then the graded
  character of $R/I$ is
  \begin{equation*}
    \bigchi_{R/I} = \bigchi_R
    \left(\sum_{u=0}^{n-1} (-1)^u \bigchi^{(n-u,1^u)} t^{du} \right)
    \prod_{i=1}^m (1-t^{c_i}).
  \end{equation*}
\end{theorem}

Note that $0\leq m\leq 1$ by Proposition \ref{pro:1}.

\begin{proof}
  Let $V$ be an equivariant lift of $I/\mathfrak{m}I$ to $R$.  In this
  case, $V = \langle g_1, \ldots, g_{n-1}, f_1,\ldots,f_m\rangle$,
  where $W = \langle g_1, \ldots, g_{n-1}\rangle$ is isomorphic to the
  standard representation $S^{(n-1,1)}$, the $g_i$ are homogeneous of
  degree $d$, and $f_i$ is a homogeneous symmetric polynomial of
  degree $c_i$.

  The Koszul complex $\mathcal{K} (g)_\bullet$ on $g_1,\ldots,g_{n-1}$
  is a minimal free resolution of $R/(g_1,\ldots,g_{n-1})$ as an $R$-module.  The term of
  $\mathcal{K}(g)_\bullet$ in homological degree $u$ is
  \begin{equation*}
    \mathcal{K} (g)_u \cong R \otimes \bigwedge^u W.
  \end{equation*}
  As indicated in \cite[\S~7.3, Exercise 11(b)]{MR1464693}, one can
  show that
  \begin{equation*}
    \bigwedge^u S^{(n-1,1)} \cong S^{(n-u,1^u)}.
  \end{equation*}
  Hence $\bigwedge^u W$ is isomorphic to a copy of $S^{(n-u,1^u)}$
  concentrated in degree $du$.  It follows that the graded character
  of $\bigwedge^u W$ is
  \begin{equation*}
    \bigchi_{\bigwedge^u W} = \bigchi^{(n-u,1^u)} t^{du},
  \end{equation*}
  and therefore we have
  \begin{equation*}
    \bigchi_{\mathcal{K} (g)_u} =  \bigchi_R \bigchi_{\bigwedge^u W} =
    \bigchi_R \bigchi^{(n-u,1^u)} t^{du}.
  \end{equation*}

  By the additivity of characters along exact sequences, we obtain
  \begin{equation*}
    \bigchi_{R/(g_1,\ldots,g_{n-1})} = \sum_{u=0}^{n-1} (-1)^u
    \bigchi_{\mathcal{K} (g)_u} = \bigchi_R
    \left(\sum_{u=0}^{n-1} (-1)^u \bigchi^{(n-u,1^u)} t^{du} \right).
  \end{equation*}

  Now, if there is an element $f_1$, then it can be modded out as in
  the proof of Theorem \ref{thm:1}.
\end{proof}

\begin{example}
  \label{exa:4}
  Let $n=4$, so $R=\Bbbk [x_1,\ldots,x_4]$.  Consider the ideal
  $I=(x_1^2,x_2^2,x_3^2,x_4^2)$.  The generators of $I$ clearly form a
  regular sequence.  Note that
  $V = \langle x_1^2,x_2^2,x_3^2,x_4^2 \rangle$ is
  $\mathfrak{S}_4$-stable.  Moreover, $V$ is isomorphic to the
  permutation representation of $\mathfrak{S}_4$, which is isomorphic
  to $S^{(4)} \oplus S^{(3,1)}$ (see \cite[Examples 1.9.5,
  2.3.8]{MR1824028}).

  Using Theorem \ref{thm:3}, we obtain
  \begin{equation*}
    \begin{split}
      \bigchi_{R/I} &= \bigchi_R \left( \sum_{u=0}^3 (-1)^u
      \bigchi^{(4-u,1^u)} t^{3u} \right) (1-t^2) =\\
      &=\frac{\bigchi_{R_{\mathfrak{S}_4}}}{\prod_{i=1}^4 (1-t^i)}
      (1-t^2) \sum_{u=0}^3 (-1)^u \bigchi^{(4-u,1^u)} t^{3u}=\\
      &=\bigchi^{(4)} + (\bigchi^{(4)} + \bigchi^{(3,1)})t + (\bigchi^{(4)} + \bigchi^{(3,1)} + \bigchi^{(2,2)})t^2\\
      &\quad + (\bigchi^{(4)} + \bigchi^{(3,1)})t^3 + \bigchi^{(4)}
      t^4.
    \end{split}
  \end{equation*}
  The computation relies on the formula for
  $\bigchi_{R_{\mathfrak{S}_4}}$ in equation \eqref{eq:5}, and
  requires multiplying irreducible characters of
  $\mathfrak{S}_4$. This can be done using the character table of
  $\mathfrak{S}_4$, which we provide in Appendix
  \ref{sec:character-table-S4} for the convenience of the reader.  
\end{example}

\subsection{Case IV: one square representation}
\label{sec:case-iv:-one}

\begin{lemma}
  \label{lem:1}
  The $\mathfrak{S}_4$-representation $\bigwedge^2 S^{(2,2)}$ is
  isomorphic to the alternating representation $S^{(1^4)}$.
\end{lemma}
\begin{proof}
  Since $\bigwedge^2 S^{(2,2)}$ has dimension 1, it is either trivial
  or alternating. Recall that the tableaux
  \begin{equation*}
    \ytableausetup{boxsize=normal}
    T_1 =
    \begin{ytableau}
      1 & 2 \\
      3 & 4
    \end{ytableau}
    \qquad\text{and}\qquad
    T_2 =
    \begin{ytableau}
      1 & 3 \\
      2 & 4
    \end{ytableau}
  \end{equation*}
  form a basis of $S^{(2,2)}$. Hence $T_1 \wedge T_2$ spans
  $\bigwedge^2 S^{(2,2)}$. We determine the action of the
  transposition $(1\ 2)$ on $T_1$ and $T_2$ using the straightening
  algorithm described in \cite[\S~7.4]{MR1464693}. We get
  \begin{align*}
    &(1\ 2) T_1
      =
      \begin{ytableau}
        2 & 1 \\
        3 & 4
      \end{ytableau}
            =
            \begin{ytableau}
              1 & 2 \\
              3 & 4
            \end{ytableau}
                  +
                  \begin{ytableau}
                    2 & 3 \\
                    1 & 4
                  \end{ytableau}
                        = T_1 - T_2,\\
    &(1\ 2) T_2
      =
      \begin{ytableau}
        2 & 3 \\
        1 & 4
      \end{ytableau}
            =- T_2.
  \end{align*}
  Therefore
  \begin{multline*}
    (1\ 2) (T_1 \wedge T_2) = ((1\ 2) T_1) \wedge ((1\ 2) T_2) =\\
    =(T_1 - T_2 )\wedge (-T_2) = - T_1 \wedge T_2.
  \end{multline*}
  We conclude that $\bigwedge^2 S^{(2,2)}$ is not trivial.
\end{proof}

\begin{theorem}
  \label{thm:4}
  Let $R = \Bbbk [x_1,\ldots,x_4]$.  Let $I\subseteq R$ be an
  $\mathfrak{S}_4$-stable complete intersection ideal.  Assume
  $I/\mathfrak{m}I$ is isomorphic to a direct sum of one copy of
  $S^{(2,2)}$ in degree $d$ and $m$ trivial representations in degrees
  $c_1,\ldots,c_m$. Then the graded character of $R/I$ is
  \begin{equation*}
    \bigchi_{R/I} = \bigchi_R
    (\bigchi^{(4)} - \bigchi^{(2,2)} t^d + \bigchi^{(1^4)} t^{2d})
    \prod_{i=1}^m (1-t^{c_i}).
  \end{equation*}
\end{theorem}

Note that $0\leq m\leq 2$ by Proposition \ref{pro:1}.

\begin{proof}
  Let $V$ be an equivariant lift of $I/\mathfrak{m}I$ to $R$.  In this
  case, $V = \langle g_1, g_2, f_1,\ldots,f_m\rangle$,
  where $W = \langle g_1, g_2\rangle$ is
  isomorphic to $S^{(2,2)}$, $g_1$ and $g_2$ are homogeneous
  of degree $d$, and $f_i$ is a homogeneous symmetric
  polynomial of degree $c_i$.

  The Koszul complex $\mathcal{K} (g)_\bullet$ on the elements
  $g_1,g_2$ is a minimal free resolution of $R/(g_1,g_2)$ as an
  $R$-module.  The terms of $\mathcal{K}(g)_\bullet$ are
  \begin{align*}
    &\mathcal{K} (g)_0 \cong R \otimes \bigwedge^0 W \cong R \otimes S^{(4)},\\
    &\mathcal{K} (g)_1 \cong R(-d) \otimes \bigwedge^1 W \cong R(-d) \otimes S^{(2,2)},\\
    &\mathcal{K} (g)_2 \cong R(-2d) \otimes \bigwedge^2 W \cong R(-2d) \otimes S^{(1^4)},
  \end{align*}
  where the last line follows from Lemma \ref{lem:1}. By the
  additivity of characters along exact sequences, we obtain
  \begin{equation*}
    \bigchi_{R/(g_1,g_2)} = \sum_{u=0}^{2} (-1)^u \bigchi_{\mathcal{K} (g)_u}
    = \bigchi_R (\bigchi^{(4)} - \bigchi^{(2,2)} t^d + \bigchi^{(1^4)} t^{2d}).
  \end{equation*}

  Now any element $f_i$ can be modded out as in the proof of Theorem
  \ref{thm:1}.
\end{proof}

\begin{example}
  \label{exa:5}
  The polynomials
  \begin{equation*}
    g_1 = (x_1-x_2)(x_3-x_4), \quad g_2 = (x_1-x_3)(x_2-x_4)
  \end{equation*}
  span a copy of $S^{(2,2)}$ in $R$. In fact, this is an example of
  the construction of the irreducible representations of
  $\mathfrak{S}_n$ due to Specht (see \cite[\S~7.4, Exercise
  17]{MR1464693}).

  Consider the ideal $I = (g_1,g_2,e_2,e_1^3)$. It can be verified,
  either by hand or using computer algebra software, that
  $g_1,g_2,e_2,e_1^3$ is a regular sequence. Using Theorem
  \ref{thm:4}, we obtain
  \begin{equation*}
    \begin{split}
      \bigchi_{R/I} &= \bigchi_R (\bigchi^{(4)} - \bigchi^{(2,2)} t^d + \bigchi^{(1^4)} t^{2d}) (1-t^2) (1-t^3) =\\
      &=\frac{\bigchi_{R_{\mathfrak{S}_4}}}{\prod_{i=1}^4 (1-t^i)}
      (1-t^2) (1-t^3) (\bigchi^{(4)} - \bigchi^{(2,2)} t^d
      + \bigchi^{(1^4)} t^{2d})=\\
      &= \bigchi^{(4)} + (\bigchi^{(4)} + \bigchi^{(3,1)})t +
      (\bigchi^{(4)}+2\bigchi^{(3,1)})t^2 + (\bigchi^{(4)}+2\bigchi^{(3,1)})t^3\\
      &\quad + (\bigchi^{(4)}+\bigchi^{(3,1)})t^4 + \bigchi^{(4)}t^5.
    \end{split}
  \end{equation*}
\end{example}

\subsection{A note on socles}
\label{sec:note-socles}

If $I\subseteq R$ is an artinian ideal, the \emph{socle} of $R/I$ is
defined as
\begin{equation*}
  \{ \bar{r} \in R/I \mid \forall i \in \{1,\ldots,n\}, x_i r \in I\}.
\end{equation*}
In the case of a complete intersection ideal $I$, the quotient $R/I$
is Gorenstein (cf. \cite[Corollary 21.19]{MR1322960}). In particular,
the socle of $R/I$ is a $\Bbbk$-vector space of dimension one (see
\cite[Proposition 21.5]{MR1322960}).  In fact, the socle will be the
graded component of $R/I$ of the highest degree.  When $I$ is
$\mathfrak{S}_n$-stable, so is the socle. Hence it is natural to ask
whether the socle of $R/I$ is a trivial or alternating representation.

The ideals in the examples of this section are all artinian. In
Example \ref{exa:2} the ideal is generated by symmetric polynomials
and the socle is alternating. In Examples \ref{exa:3}, \ref{exa:4}, and
\ref{exa:5} the ideal has generators that are not symmetric and the
socle is trivial.

In fact, the socle of $R/I$ is alternating if and only if $I$ is
generated by symmetric polynomials. Although it may be possible to
prove this fact using the character formulas we presented, this result
is a special case of \cite[Theorem 3.2]{MR1779775} and \cite[Theorem
3.1]{MR2134302}.

\begin{appendix}
  \section{The character table of \texorpdfstring{$\mathfrak{S}_4$}{S4}}
  \label{sec:character-table-S4}
  \begin{center}
    \begin{tabulary}{0.75\textwidth}{K{1.2cm}|K{1.6cm}K{1.6cm}K{1.6cm}K{1.6cm}K{1.6cm}}
      & 1 & 6 & 8 & 6 & 3\\
      & 1 & (1\ 2) & (1\ 2\ 3) & (1\ 2\ 3\ 4) & (1\ 2)(3\ 4)\\ \hline
      $\bigchi^{(4)}$ & 1 & 1 & 1 & 1 & 1\\
      $\bigchi^{(3,1)}$ & 3 & 1 & 0 & -1 & -1\\
      $\bigchi^{(2,2)}$ & 2 & 0 & -1 & 0 & 2\\
      $\bigchi^{(2,1^2)}$ & 3 & -1 & 0 & 1 & -1\\
      $\bigchi^{(1^4)}$ & 1 & -1 & 1 & -1 & 1\\
    \end{tabulary}
  \end{center}
\end{appendix}


\end{document}